\pdfoutput=1
\RequirePackage{ifpdf}
\ifpdf 
\documentclass[pdftex]{sigma}
\else
\documentclass{sigma}
\fiumentclass{sigma}
\fi

\usepackage{mathrsfs}
\usepackage{enumitem}

\numberwithin{equation}{section}

\newtheorem{Theorem}{Theorem}[section]
\newtheorem*{Theorem*}{Theorem}
\newtheorem{Corollary}[Theorem]{Corollary}
\newtheorem{Lemma}[Theorem]{Lemma}
\newtheorem{Proposition}[Theorem]{Proposition}
\newtheorem{obs}[Theorem]{Observation}
 { \theoremstyle{definition}
\newtheorem{Definition}[Theorem]{Definition}

\newtheorem{exas}[Theorem]{Examples}
\newtheorem{exa}[Theorem]{Example}}

\DeclareMathOperator{\ad}{ad}

\DeclareMathOperator{\diver}{div}
\DeclareMathOperator{\dom}{Dom}
\DeclareMathOperator{\ess}{ess}
\DeclareMathOperator{\essran}{essran}

\DeclareMathOperator{\id}{id}

\DeclareMathOperator{\ran}{Ran}

\DeclareMathOperator{\supp}{supp}
\DeclareMathOperator{\tr}{tr}

\def\p{\mathrm{pp}}

\def\d{\mathrm d}
\def\ac{\mathrm{ac}}
\def\sc{\mathrm{sc}}
\def\s{\mathrm s}
\def\C{\mathbb C}
\def\F{\mathbb F}

\def\R{\mathbb R}

\def\ve{\varepsilon}
\def\vf{\varphi}

\begin{document}
\allowdisplaybreaks

\renewcommand{\thefootnote}{}

\newcommand{\arXivNumber}{2301.10986}

\renewcommand{\PaperNumber}{050}

\FirstPageHeading

\ShortArticleName{On the Spectrum of Certain Hadamard Manifolds}

\ArticleName{On the Spectrum of Certain Hadamard Manifolds\footnote{This paper is a~contribution to the Special Issue on Differential Geometry Inspired by Mathematical Physics in honor of Jean-Pierre Bourguignon for his 75th birthday. The~full collection is available at \href{https://www.emis.de/journals/SIGMA/Bourguignon.html}{https://www.emis.de/journals/SIGMA/Bourguignon.html}}}

\Author{Werner BALLMANN~$^{\rm a}$, Mayukh MUKHERJEE~$^{\rm b}$ and Panagiotis POLYMERAKIS~$^{\rm c}$}

\AuthorNameForHeading{W.~Ballmann, M.~Mukherjee and P.~Polymerakis}

\Address{$^{\rm a)}$~Max Planck Institute for Mathematics, Vivatsgasse 7, 53111~Bonn, Germany}
\EmailD{\href{mailto:hwbllmnn@mpim-bonn.mpg.de}{hwbllmnn@mpim-bonn.mpg.de}}

\Address{$^{\rm b)}$~Indian Institute of Technology Bombay, Powai, 400076~Maharashtra, India}
\EmailD{\href{mailto:mathmukherjee@gmail.com}{mathmukherjee@gmail.com}}

\Address{$^{\rm c)}$~Department of Mathematics, University of Thessaly, 3rd km Old National Road Lamia-Athens,\\
\hphantom{$^{\rm c)}$}~35100 Lamia, Greece}
\EmailD{\href{mailto:ppolymerakis@uth.gr}{ppolymerakis@uth.gr}}

\ArticleDates{Received January 27, 2023, in final form July 14, 2023; Published online July 23, 2023}

\Abstract{We show the absolute continuity of the spectrum and determine the spectrum as a set for two classes of Hadamard manifolds and for specific domains and quotients of one of the classes.}


\Classification{58J50; 53C20}

\begin{flushright}
\begin{minipage}{65mm}
\it Dedicated to Jean-Pierre Bourguignon\\ on the occasion of his 75th birthday
\end{minipage}
\end{flushright}

\renewcommand{\thefootnote}{\arabic{footnote}}
\setcounter{footnote}{0}

\section{Introduction}

The spectrum of the Laplacian is a classical invariant in Riemannian geometry.
If the underlying manifold is closed,
then the spectrum of its Laplacian consists of eigenvalues of finite multiplicity,
and many studies are concerned with estimates of the eigenvalues and their multiplicities.
In this paper, we investigate the spectrum of (the Laplacian of) non-compact Riemannian manifolds.
Then the structure of the spectrum is quite different in general.
For example, the spectrum of Euclidean spaces does not have eigenvalues,
but is, what is called absolutely continuous.

For a complete and connected Riemannian manifold $M$, we view its Laplacian $\Delta$
as an unbounded and symmetric operator with domain $C^\infty_{\rm c}(M) \subseteq L^2(M)$.
The closure of $\Delta$, which we also denote by $\Delta$, is a self-adjoint operator in $L^2(M)$.
Its spectrum as a subset of $\R$ will be denoted by $\sigma(M)$.

Functional analysis of self-adjoint operators yields the orthogonal decomposition
\begin{align*}
L^2(M) = H_{\p}(M) \oplus H_{\ac}(M) \oplus H_{\sc}(M)
\end{align*}
of $L^2(M)$ into $\Delta$-invariant subspaces such that
\begin{enumerate}\itemsep=0pt
\item[(1)] $H_{\p}(M)$ is spanned by eigenfunctions;
\item[(2)] $H_{\ac}(M)$ consists of functions with absolutely continuous spectral measure;
\item[(3)] $H_{\sc}(M)$ consists of functions with singularly continuous spectral measure.
\end{enumerate}
We present some more details about these notions and decompositions in Appendix~\ref{secspec},
since they are needed in our discussion of spectra of Riemannian products.

There is the well known decomposition $\sigma(M)=\sigma_{\d}(M)\cup\sigma_{\ess}(M)$ of the spectrum
as a~disjoint union of \emph{discrete} and \emph{essential spectrum},
where $\lambda\in\R$ belongs to $\sigma_{\ess}(M)$ if $\Delta-\lambda$ is not a~Fredholm operator.
Then $\lambda\in\sigma_{\d}(M)$ if and only if $\lambda$ is an eigenvalue of $\Delta$
of finite multiplicity and is an isolated point of $\sigma(M)$.
Therefore, the preimage in $L^2(M)$ of $\sigma_{\d}(M)$ under the spectral projection associated to $\Delta$
is contained in $H_{\p}(M)$.
In particular, if the essential spectrum of~$\Delta$ is empty, then $H_{\ac}(M)=H_{\sc}(M)=\{0\}$.
The essential spectrum of $\Delta$ is determined by the geometry of $M$ at infinity;
see, for example,~\cite[Proposition 3.6]{BMM17} for a general formulation of this fact.
In particular, the essential spectrum vanishes if $M$ is compact.
But there are also a~number of results about the vanishing of the essential spectrum
for non-compact manifolds;
see, for example,~\cite[Theorems 1.1 and 1.2]{BLM10} and~\cite[Examples 3.7]{BMM17}.

Now the vanishing of the essential spectrum implies the vanishing of the absolutely continuous spectrum,
which is not what we are aiming for.
In fact, we will obtain conditions which imply the absolute continuity of the spectrum, $H_{\ac}(M)=L^2(M)$,
or, in other words, the vanishing of the point and the singular continuous spectrum.
As a byproduct, we will also discuss the vanishing of the point spectrum, $H_{\p}(M)=\{0\}$.
The underlying manifolds will be \emph{Hadamard manifolds}, that is,
complete and simply connected Riemannian manifolds of non-positive sectional curvature.

Recall that a homogeneous Hadamard manifold can be thought of as a simply connected solvable Lie group $S$
with a left-invariant Riemannian metric
and given as a semi-direct product $S=A\ltimes N$, where $A$ is Abelian and $N$ is nilpotent.
Our main result is the following

\begin{Theorem}\label{onlyach}
If $M$ is a homogeneous Hadamard manifold,
then the spectrum of $M$ is absolutely continuous with $\sigma(M)=\bigl[h^2/4,\infty\bigr)$,
where $h$ is the mean curvature of $N$ in $S\cong M$.
\end{Theorem}

An explicit formula for $h$ is given in Theorem~\ref{specsn},
where we also show that the equality $\sigma(M)=\bigl[h^2/4,\infty\bigr)$ holds in greater generality,
namely for left-invariant Riemannian metrics on Lie groups $G$,
which are solvable or compact extensions of solvable Lie groups,
such that their derived subgroups are not cocompact in $G$.

Concerning the identification of homogeneous Hadamard manifolds
with certain simply connected solvable Lie groups with left-invariant Riemannian metrics,
one may ask, whether representation theory could be a tool to study the spectrum of their Laplacians.
However, in general, in our situation the underlying Lie groups are not exponential solvable,
see~\cite[p.~26 and p.~37]{AzencottWilsonI76} and~\cite[Remark (ii) on p.~16]{AzencottWilsonII76},
so that at least the extension of Kirillov theory by Auslander and Kostant does not seem to apply.

For Euclidean and symmetric spaces of non-compact type,
there is much more precise information on their spectrum.
For example, the Fourier transform yields a unitary equivalence between the Laplacian in $L^2(\R^m)$
and multiplication by $|x|^2$ in $L^2(\R^m)$,
which implies that the spectrum of the Laplacian is absolutely continuous with spectrum $[0,\infty)$.
Similar, but more involved, characterizations of the Laplacian are known
in the case of symmetric spaces of non-compact type.

We say that a Hadamard manifold $M$ is \emph{asymptotically harmonic about a point $\xi\in M_\infty$}
if there is a number $h\ge0$ such that each horosphere in $M$ with center $\xi$ has constant mean curvature $h$.
Hyperbolic spaces $H^k_{\F}$, endowed with Fubini--Study metrics, are asymptotically harmonic about any point in $M_\infty$, where $\lambda_0=h^2/4>0$.
Asymptotically harmonic manifolds in the usual sense are asymptotically harmonic about any point in $M_\infty$.
Such manifolds arise, for example, in investigations on the rigidity of geodesic flows.
Furthermore, since horospheres are limits of spheres,
harmonic Hadamard manifolds are asymptotically harmonic about any point in $M_\infty$.
However, homogeneous Hadamard manifolds without Euclidean factor,
which are not hyperbolic spaces, are asymptotically harmonic about some, but not any point $M_\infty$.
A second main result of this article is

\begin{Theorem}\label{onlyacm}
If $K_M<0$ and $M$ is asymptotically harmonic about a point $\xi\in M_\infty$,
then the spectrum of $\Delta$ on $M$ is absolutely continuous with $\sigma(M)\subseteq\bigl[h^2/4,\infty\bigr)$,
where $h>0$ is the mean curvature of the horospheres with center $\xi$.
Moreover, if the sectional curvature of $M$ is negatively pinched and the covariant derivative of the curvature tensor of $M$ is uniformly bounded, then $\sigma(M)=\bigl[h^2/4,\infty\bigr)$.
\end{Theorem}

Say that a group of isometries of a Hadamard manifold $M$ is \emph{elementary}
if it is discrete and without torsion, fixes a point $\xi\in M_\infty$,
and such that it is of one of the following two types:
\begin{enumerate}[label=(\alph*)]\itemsep=0pt
\item\label{ax} $\Gamma$ leaves a geodesic $\gamma$ in $M$ emanating from $\xi$ invariant.
\item\label{pa} $\Gamma$ leaves Busemann functions associated to $\xi$ invariant.
\end{enumerate}
In the first case, $C=\Gamma\backslash\gamma$ is a circle and the projection $M\to\gamma$
along horospheres with center~$\xi$ descends to a Riemannian submersion $\pi\colon N\to C$,
where $N=\Gamma\backslash M$.
In the second case,
any Busemann function $b$ associated to $\xi$ is invariant under $\Gamma$
and hence pushes down to $N$ and defines a Riemannian submersion $\pi=b\colon N\to\R$.

\begin{Theorem}\label{onlyacn}
For $M$ as in Theorem~{\rm \ref{onlyacm}}, if $\Gamma$ is an elementary group of isometries of $M$ fixing~$\xi$,
then the spectrum of the quotient $N=\Gamma\backslash M$ is absolutely continuous
with $\sigma(N)\subseteq\bigl[h^2/4,\infty\bigr)$, where $h>0$ denotes the mean curvature of the fibers of $\pi$.
Moreover,
\begin{enumerate}[label={\rm (\arabic*)}]\itemsep=0pt
\item
if $N$ is of type {\rm \ref{pa}} and the fibers of $\pi$ are of finite volume or
\item
if $N$ is of type {\rm \ref{pa}}, the fibers of $\pi$ are of infinite volume,
the sectional curvature of $M$ is negatively pinched,
and the covariant derivative of the curvature tensor of $M$ is uniformly bounded,
\end{enumerate}
then $\sigma(N)=\bigl[h^2/4,\infty\bigr)$.
\end{Theorem}

Since Euclidean space $\R^m$ has vanishing curvature and a basis of parallel differential forms,
its Hodge--Laplacian $({\rm d}+{\rm d}^*)^2$ decomposes into parts which are unitarily equivalent to its Laplacian on functions.
In particular, its differential form spectrum is absolutely continuous.
One might ask whether the differential form spectrum of Hadamard manifolds as considered above
is also absolutely continuous.
Except for symmetric spaces,
we do not know of any such general result.
But there are results about the essential spectrum of the Hodge--Laplacian,
given sufficiently strong pinching of sectional curvature (depending on dimension and degree),
see for ex\-am\-ple~\cite{DX84,Kas94},
and, for asymptotically hyperbolic manifolds, on the non-existence of eigenvalues
above the essential spectrum, see~\cite[Theorem 16]{Ma91}.

Regarding the absolute continuity of the spectrum,
we rely on the same integral formulae as the ones used by Xavier~\cite{Xavier1988} and Donnelly--Garofalo~\cite{DoGa92,DoGa97},
using, in particular, arguments from~\cite{Xavier1988}.
One important difference to the latter is that we also use functions which are convex,
but not strictly convex.
The determination of the spectra as sets is by different arguments,
where we rely on~\cite{P21} in the case of homogeneous Hadamard manifolds.

\section{Preliminaries}

For a Hadamard manifold $M$ and points $\xi\in M_\infty$ and $x\in M$,
the Busemann function $b$ associated to $\xi$ with $b(x)=0$ is given by
\begin{align*}
	b(y) = \lim_{n\to\infty} (d(x_n,y)-d(x_n,x)),
\end{align*}
where $(x_n)$ is any sequence in $M$ converging to $\xi$.
Recall that the limit exists and that $b$ is a $C^2$ distance function on $M$
with horoballs and horospheres centered at $\xi$ as sublevel and level sets;
cf.~\cite{Ba95} for this and other facts about Hadamard manifolds.

\begin{obs}\label{chee}
If a Hadamard manifold $M$ is asymptotically harmonic about a point $\xi\in M_\infty$
or if $N=\Gamma\backslash M$ is of type {\rm \ref{pa}},
then the Cheeger constant of $M$ respectively $N$ is at least $h$,
where~$h$ is the mean curvature of the horospheres with center $\xi$.
In particular, the spectrum of~$M$ respectively $N$ is contained in $\bigl[h^2/4,\infty\bigr)$.
\end{obs}

Throughout the article,
we use the absolute value sign to denote the appropriate volumes of sets under discussion.

\begin{proof}[Proof of Observation~\ref{chee}]
Let $D$ be a compact domain in $M$ respectively $N$ with smooth boundary.
Let $b$ be a Busemann function associated to $\xi$ and $X=\nabla b$.
Then $|X|=1$ and $\diver X=h$ and hence
\begin{align*}
	h|D| = \int_D\diver X = \int_{\partial D}\langle X,\nu\rangle \le |\partial D|,
\end{align*}
where $\nu$ is the outer normal field of $D$ along $\partial D$.
The last assertion is just the Cheeger in\-equality.
\end{proof}

\begin{obs}\label{obs}
If a Hadamard manifold $M$ is asymptotically harmonic about a point $\xi\in M_\infty$
and $b$ is a Busemann function of $M$ centered at $\xi$, then
\begin{align*}
	0 \le \nabla^2b(v,v) \le h |v|^2.
\end{align*}
\end{obs}

\begin{proof}
The claim follows immediately from $h=\tr\nabla^2b$ and $\nabla^2b\ge0$.
\end{proof}

Suppose that a Lie group $G$ acts properly, freely, and isometrically on a complete Riemannian manifold $M$.
Then the quotient $Q=G\backslash M$ is a smooth manifold.
Moreover, $Q$ inherits a~Riemannian metric such that the projection $\pi\colon M\to Q$ is a~Riemannian submersion.
The mean curvature field $H$ of the fibers is $\pi$-related to a vector field on $Q$,
which is then the push-forward~$\pi_*H$ of $H$.
Following~\cite{P21}, we define a Schr\"odinger operator on $Q$,
\begin{align}\label{pdef}
	S = \Delta + \frac14|\pi_*H|^2 - \frac12\diver\pi_*H.
\end{align}
For the determination of the spectrum $\sigma(M)$ of a homogeneous Hadamard manifold $M$,
we will then invoke~\cite[Theorem 1.3]{P21} in the following form.

\begin{Theorem}\label{p13}
Suppose that $G$ is amenable and that its component of the identity is unimodular.
Then
\begin{align*}
	\lambda_0(M) = \lambda_0(S) =:\lambda_0
	\qquad \text{and} \qquad
	\sigma(S)\subseteq\sigma(M),
\end{align*}
where $\lambda_0(M)$ and $\lambda_0(S)$ denote the bottom of the spectrum $\sigma(M)$ of $M$
and $\sigma(S)$ of $S$, respectively.
In particular, $\sigma(S)=[\lambda_0,\infty)$ implies that $\sigma(M)=[\lambda_0,\infty)$.
\end{Theorem}

For a boundary condition $B$ of a domain $D$ in a Riemannian manifold $M$,
denote by $C^\infty_{c,B}(\bar{D})$ the space of smooth functions on $\bar D$ with compact support
which satisfy $B$ along the boundary~$\partial D$ of $D$.

\begin{Definition}
We say that a boundary condition $B$ for a smooth domain $D$ is \emph{non-positive}
if~$u\in C^\infty_{c,B}(\bar D)$ implies $2u\nabla u=\nabla_\nu u^2\le0$ along $\partial D$.
\end{Definition}

\begin{exas}
Dirichlet and Neumann boundary conditions are non-positive.
Robin boundary conditions $\alpha u+\beta\nabla_\nu u=0$ are non-positive if $\alpha\beta>0$.
These kinds of boundary conditions are also self-adjoint and elliptic in the usual sense.
\end{exas}

\section{Vanishing of point spectrum}\label{secrell}

For a $C^1$ vector field $X$ and a $C^2$ function $u$ on a Riemannian manifold $M$, we have the following identity
\begin{align}
 2\langle\nabla_{\nabla u}X,\nabla u\rangle
 = 2\langle X,\nabla u\rangle\Delta u + |\nabla u|^2\diver X + \diver\big\{2X(u)\nabla u-|\nabla u|^2X\big\}. \label{rell}
\end{align}
In the context of spectral theory, this identity was used by Donnelly--Garofalo,
see~\cite[Lemma~2.1]{DoGa92}, \cite[Proposition 3.1]{DoGa97}, and Xavier, see~\cite[identity~(1)]{Xavier1988},
but it also occurs in the earlier work of Kazdan--Warner~\cite[identity~(8.1)]{KazdanWarner74c}.
Donnelly and Garofalo attribute \eqref{rell} to Rellich~\cite{Rellich43} in the case of Euclidean spaces;
cf.~(3) respectively II on page 62 of~\cite{Rellich43}.
For the convenience of the reader, we give a~short proof of the identity.

\begin{proof}[Proof of \eqref{rell}]
Since $X(u)=\langle\nabla u, X\rangle$ and $|\nabla u|^2=\langle\nabla u,\nabla u\rangle=\nabla u(u)$, we have
\begin{align*}
 2\diver(X(u)\nabla u)
 &= 2\nabla u(X(u)) + 2X(u)\diver\nabla u \\
 &= 2X(\nabla u(u)) + 2[\nabla u,X](u) - 2X(u)\Delta u \\
 &= 2X\bigl(|\nabla u|^2\bigr) + 2(\nabla_{\nabla u}X)(u) - 2(\nabla_X\nabla u)(u) - 2X(u)\Delta u \\
 &= 2X\bigl(|\nabla u|^2\bigr) + 2\langle\nabla_{\nabla u}X,\nabla u\rangle - 2\langle\nabla_X\nabla u,\nabla u\rangle - 2X(u)\Delta u \\
 &= X\bigl(|\nabla u|^2\bigr) + 2\langle\nabla_{\nabla u}X,\nabla u\rangle - 2X(u)\Delta u \\
 &= \diver\bigl(|\nabla u|^2X\bigr) - |\nabla u|^2\diver X + 2\langle\nabla_{\nabla u}X,\nabla u\rangle - 2X(u)\Delta u,
\end{align*}
which is the assertion.
\end{proof}

\begin{Corollary}
If $\Delta u=\varphi u$ for a $C^1$ function $\varphi$ on $M$, then
\begin{gather}
 2\langle\nabla_{\nabla u}X,\nabla u\rangle
= \varphi X(u^2) + |\nabla u|^2\diver X + \diver\big\{2X(u)\nabla u-|\nabla u|^2X\big\} \label{doga2} \\
\hphantom{2\langle\nabla_{\nabla u}X,\nabla u\rangle}{}= \bigl(|\nabla u|^2 - \varphi u^2\bigr)\diver X - X(\varphi)u^2 \notag\\
 \hphantom{2\langle\nabla_{\nabla u}X,\nabla u\rangle=}{}+ \diver\big\{2X(u)\nabla u+\bigl(\varphi u^2-|\nabla u|^2\bigr)X\big\}. \label{doga3}
\end{gather}
\end{Corollary}

\begin{proof}
If $\Delta u=\varphi u$, then
\begin{align*}
 2X(u)\Delta u &= 2\varphi X(u)u = \varphi X\bigl(u^2\bigr) = X\bigl(\varphi u^2\bigr) - X(\varphi)u^2 \\
 &= \diver\bigl(\varphi u^2X\bigr) - \varphi u^2\diver X - X(\varphi)u^2.
\end{align*}
Together with \eqref{rell}, this gives \eqref{doga2} and \eqref{doga3}.
\end{proof}

In~\cite{DoGa97}, Donnelly and Garofalo consider eigenfunctions $Lu=\lambda u$ of Schr\"odinger operators $L=\Delta+V$.
This corresponds to $\varphi=\lambda-V$ in \eqref{doga2} and \eqref{doga3}.

To exhibit the strength of \eqref{rell}, we present immediate applications of the above identities.
With somewhat more elaborate techniques, we will obtain stronger results later.
Our results are reminiscent of~\cite[Satz 2]{Rellich43}.

\begin{Theorem}\label{nopointd}
Suppose that $K_M<0$ and that $M$ is asymptotically harmonic about a point $\xi\in M_\infty$.
Let $D$ be the complement of a horoball in $M$ with center $\xi$.
Then the point spectrum $H_{B,\p}=H_{B,\p}(\Delta, \bar{D} )$ vanishes for any self-adjoint elliptic
boundary condition $B$ for $D$, which is non-positive.
\end{Theorem}

\begin{proof}
Let $b$ be the Busemann function associated to $\xi$ such that $D =b^{-1}((0,\infty))$.
By assumption, there is a number $h$, because of negative curvature strictly positive,
such that the mean curvature of the horospheres $b^{-1}(r)$ is constant, equal to $h$.
In other words, $\diver\nabla b=h$.
Let
\begin{align}\label{nopointd2}
X = \bigl(1 - {\rm e}^{-hb}\bigr)\nabla b.
\end{align}
Then
\begin{align*}
 \diver X=h \qquad\text{and}\qquad
 \nabla X = h{\rm e}^{-hb} \nabla b\otimes\nabla b + \bigl(1 - {\rm e}^{-hb}\bigr)\nabla^2b.
\end{align*}
Now $\nabla^2b$ vanishes in the direction of $\nabla b$ and is positive definite perpendicular to it.
Since $1-{\rm e}^{-hb}>0$ on $D$, we conclude that $\nabla X$ is positive definite on $D$.
By Observation~\ref{obs}, $\nabla^2b$ is bounded on $M$, hence $\nabla X$ on $D$.

Let $u$ be a square-integrable smooth function on $D $ with $\Delta u=\lambda u$,
satisfying the boundary condition $B$.
Since $X$ vanishes on $\partial D=b^{-1}(0)$ and $\nabla X$ is bounded on $D$,
integration of \eqref{doga3}, with $\vf=\lambda$, is justified and implies that
\begin{align}
 2\langle\nabla_{\nabla u}X,\nabla u\rangle_2
 &= h\int_D \bigl(|\nabla u|^2 - \lambda u^2\bigr) \nonumber\\
 &= h\int_D \bigl(u\Delta u - \lambda u^2\bigr) + h\int_{\partial D }u\nabla_\nu u \nonumber\\
 &= h\int_{\partial D }u\nabla_\nu u \le 0,\label{nopointd6}
\end{align}
where the inequality is implied by the boundary condition $B$.
Now $\nabla X$ is positive definite on $D $, therefore $\nabla u$ vanishes, and hence $u$ is constant.
On the other hand, the volume of $D $ is infinite and $u$ is square-integrable, thus $u$ vanishes.
\end{proof}

Restricting to Dirichlet or Neumann boundary conditions, we can allow for more general kinds of domains.
Let $M$ be a Hadamard manifold.\
Say that a subset $D\subseteq M$ is a \emph{shadow} if there is a horosphere $H\subseteq M$
and a subset $C\subseteq H$ such that $D$ is the set of all points $x\in M$
such that the ray from $x$ to the center $\xi$ of $H$ passes through $C$.
Then we also say that $D$ is the \emph{shadow of $C$}, thinking of $\xi$ as a source of light.
The exterior of a horoball considered in Theorem~\ref{nopointd} above corresponds to $C=H$.

\begin{Theorem}\label{nopointd8}
Suppose that $K_M<0$ and that $M$ is asymptotically harmonic about a point $\xi\in M_\infty$.
Let $D\subseteq M$ be the shadow of a smooth domain $C$ in a horosphere $H$ in $M$ with center~$\xi$.
Then the Dirichlet and Neumann point spectra of $D$ vanish.
\end{Theorem}

\begin{proof}
Let $b$ be the Busemann function on $M$ associated to $\xi$ which vanishes on $H = b^{-1}(0)$
and $X$ be the vector field as in \eqref{nopointd2}.
As in the previous proof, we get that $\nabla X$ is positive definite on $b^{-1}((0,\infty))$,
and hence, in particular, on $D \setminus C$.

Let $u$ be a square-integrable smooth function on $D $ with $\Delta u=\lambda u$,
satisfying the Dirichlet or the Neumann boundary condition.
Now the boundary of $D$ consists of the smooth domain $C\subseteq H$
and the shadow of $\partial C$, where $X$ is tangential to it.
Since $C$ is smooth, $\partial C$ is the singular part of $\partial D$.
Now the codimension of $\partial C$ in $M$ is two,
and hence we can apply the divergence formula when integrating \eqref{doga2}.
Since $X$ vanishes on $H$,
the contribution of $C$ vanishes in the integrated version of \eqref{doga3}.
As for the shadow part, the contribution of the first boundary term $2X(u)\nabla_\nu u$ vanishes
since then $X(u)=0$ in the case of the Dirichlet boundary condition
and $\nabla_\nu u=0$ in the case of the Neumann boundary condition.
The contribution of the second term, $\bigl(\lambda u^2-|\nabla u|^2\bigr)X$,
always vanishes since $\langle X,\nu\rangle=0$ in the shadow part.
Thus repeating the computation in \eqref{nopointd6}, we get
\begin{align*}
	\langle\nabla_{\nabla u}X,\nabla u\rangle_2 = h\int_{\partial D }u\nabla_\nu u = 0.
\end{align*}
Since $\nabla X$ is positive definite on $D \setminus C$, we conclude, as above, that $u=0$.
\end{proof}

\section{Absolutely continuous spectrum}

Let $S$ be a self-adjoint operator in a separable Hilbert space $H$,
and denote by $R(z)=(S-z)^{-1}$ the resolvent of $S$.
Say that a closed operator $T$ in $H$ is \emph{$S$-smooth} if, for each $x\in H$ and $\ve\ne0$,
$R(\lambda+{\rm i}\ve)x\in\dom T$
 for almost all $\lambda\in\R$ and
\begin{align*}
 \sup_{|x|=1,\,\ve\ne0} \int_{-\infty}^\infty |TR(\lambda+{\rm i}\ve)x|^2 < \infty,
\end{align*}
see~\cite[p.~142]{ReedSimon78}.
In our context, the point is that then the image
\begin{align}\label{Kato3}
 \ran T^* \subseteq H_{\ac}(S),
\end{align}
see~\cite[Theorem XIII.23]{ReedSimon78}.
For a Borel subset $B\subseteq\R$,
say that a closed operator $T$ in $H$ is \emph{$S$-smooth on $B$} if $TE_B$ is $H$-smooth,
where $E_B$ denotes the spectral projection of $S$ associated to $B$.
Then
\begin{align*}
 \ran E_BT^* \subseteq H_{\ac}(S),
\end{align*}
by \eqref{Kato3}.
By~\cite[Theorem XIII.30]{ReedSimon78}, $T$ is $S$-smooth on the closure $\bar B$ of $B\subseteq\R$ if
\begin{align}\label{Kato7}
 \dom T\supseteq\dom S
 \qquad\text{and}\qquad
 \sup_{\substack{|x|=1, \, \lambda\in B,\\ 0<|\ve|<1}} |\ve||TR(\lambda+{\rm i}\ve)x| < \infty.
\end{align}
Xavier~\cite[p.~581f.]{Xavier1988} shows that the latter criterion is satisfied
if $B$ is bounded and if there is a~constant $C$ such that
\begin{align}\label{Kato9}
 |Tx|^2 \le C|(S-\lambda)x|(|Sx|+|x|)
\end{align}
for all $x$ in a core of $S$ in $H$ and all $\lambda\in B$.
In fact, he shows that the term in \eqref{Kato7} is then bounded by $C(1+|\lambda|+|\ve|)$,
which explains why $B$ is assumed to be bounded; see~\cite[top of~p.~582]{Xavier1988}.

Guided by~\cite[Section 3]{Xavier1988}, we return to the geometric situation in Section~\ref{secrell}
and let $X$ be a~$C^1$ vector field and $u$ be a $C^2$ function on a Riemannian manifold $M$
such that $\supp X\cap\supp u$ is compact.
Then \eqref{rell} gives, for any smooth domain $D$ in $M$ with outer normal $\nu$ along $\partial D$,
\begin{gather*}
 \int_D\langle X,\nabla u\rangle\Delta u + \frac12\int_D|\nabla u|^2\diver X \\
 \qquad\quad{}= \int_D\langle\nabla_{\nabla u}X,\nabla u\rangle
 - \int_{\partial D}\biggl(\langle X,\nabla u\rangle\langle\nabla u,\nu\rangle
 - \frac12|\nabla u|^2\langle X,\nu\rangle\biggr).
\end{gather*}
As in~\cite[p.~583]{Xavier1988}, we use
\begin{align*}
 |\nabla u|^2 = u\Delta u - \frac12\Delta u^2
\end{align*}
and obtain
\begin{align*}
 \int_D \Delta u\biggl(\langle X,\nabla u\rangle + \frac{u}2\diver X\biggr)
 = {}&\int_D \biggl(\langle\nabla_{\nabla u}X,\nabla u\rangle + \frac14\diver X\Delta u^2\biggr) \\
 &{}- \int_{\partial D}\biggl(\langle X,\nabla u\rangle\langle\nabla u,\nu\rangle
 - \frac12|\nabla u|^2\langle X,\nu\rangle\biggr).
\end{align*}
Since
\begin{align*}
 P_Xu = \langle X,\nabla u\rangle + \frac{u}2\diver X
\end{align*}
satisfies
\begin{align*}
 u P_Xu = \diver\biggl(\frac12u^2X\biggr),
\end{align*}
we arrive at
\begin{align}
 \int_D (\Delta u-\lambda u) P_Xu
 = &\int_D \biggl(\langle\nabla_{\nabla u}X,\nabla u\rangle + \frac14\diver X\Delta u^2\biggr) \notag\\
 &- \int_{\partial D}\biggl(\langle X,\nabla u\rangle\langle\nabla u,\nu\rangle
 - \frac12\bigl(|\nabla u|^2 - \lambda u^2\bigr)\langle X,\nu\rangle\biggr),\label{x1}
\end{align}
which corresponds to~\cite[identity~(2)]{Xavier1988}, but with boundary integral included.
It will also be useful to have the latter formula with one of the substitutions
\begin{align}
 \int_D\diver X\Delta u^2
 &= \int_D\big\langle\nabla\diver X,\nabla u^2\big\rangle - \int_{\partial D}\diver X\nabla_\nu u^2 \label{x2a}\\
 &= \int_D(\Delta\diver X)u^2
 + \int_{\partial D} \bigl((\nabla_\nu\diver X)u^2 - \diver X\nabla_\nu u^2\bigr), \nonumber
\end{align}
which is Green's formula applied to the functions $\diver X$ and $u^2$.
Clearly,
\begin{gather}
 \int_D (\Delta u-\lambda u) P_Xu
	\le \|\Delta u-\lambda u\|_2\|P_Xu\|_2
	\le C\|\Delta u-\lambda u\|_2(\|\Delta u\|_2+\|u\|_2)\label{x3}
\end{gather}
if $X$ and $\diver X$ are bounded on $D$,
and $u$ satisfies a non-positive boundary condition.
This will be important in view of~\eqref{Kato9}.

\begin{Theorem}
Suppose that $K_M<0$ and that $M$ is asymptotically harmonic about a point $\xi\in M_\infty$.
Let $D$ be the complement of a horoball in $M$ with center $\xi$
and $B$ be a self-adjoint elliptic boundary condition $B$ for $D$ which is non-positive.
Then the spectrum of $\Delta$ on $D$ with respect to $B$ is absolutely continuous.
\end{Theorem}

\begin{proof}
We return to the setup in the proof of Theorem~\ref{nopointd}.
Let $b$ be the Busemann function associated to $\xi$ such that $D=b^{-1}((0,\infty))$
and
\begin{align*}
 X = \bigl(1 - {\rm e}^{-hb}\bigr)\nabla b.
\end{align*}
Since $|\nabla b|=1$, $|X|<1$ on $D$.
Recall that $\diver X=h>0$ and that $\nabla X$ is positive definite on~$D$.

As in the proof of~\cite[Theorem~2]{Xavier1988}, let $Y$ be a smooth vector field on $D$ with compact support
in $D$ and $T=T_Y$ be differentiation of functions in the direction of $Y$.
Then there are constants $C_1,C_2>0$ such that $\nabla X\ge C_1$ on the support of $Y$
and such that $|Y|^2\le C_2$.

Now $C^\infty_{c,B}(\bar D)$, the space of smooth functions on $\bar D$ with compact support in $\bar D$
satisfying one of the above boundary condition $B$, is a core for $\Delta$ with boundary condition $B$.
Since $\diver X=h>0$ is constant and $X$ vanishes along $\partial D$, \eqref{x1} and \eqref{x2a} yield,
for $u\in C^\infty_{c,B}(\bar D)$,
\begin{align*}
 \int_D (\Delta u-\lambda u) P_Xu
 &= \int_D \biggl(\langle\nabla_{\nabla u}X,\nabla u\rangle + \frac{h}4\Delta u^2\biggr) \\
 &= \int_D \langle\nabla_{\nabla u}X,\nabla u\rangle - \frac{h}2\int_{\partial D}u\nabla_\nu u \\
 &\ge \int_D \langle\nabla_{\nabla u}X,\nabla u\rangle \ge C_1 \int_{\supp Y}|\nabla u|^2 \\
 &\ge \frac{C_1}{C_2}\int_D|Yu|^2 = \frac{C_1}{C_2}\|T_Yu\|_2^2.
\end{align*}
Now \eqref{x3} applies and shows that $T_Y$ is $S$-smooth on any given bounded Borel subset $A\subseteq\R$.
Hence we get from \eqref{Kato3} that $\ran E_AT_Y^*\subseteq H_{\ac}(\Delta,B)$, for any bounded Borel subset $A\subseteq\R$.
Therefore, $\ran T_Y^*\subseteq H_{\ac}(\Delta,B)$.

Now suppose that $u\in L^2(D )$ is perpendicular to $H_{\ac}(\Delta,B)$.
Then $u$ is perpendicular to $\ran T_Y^*$, for all vector fields $Y$ as above.
But then $T_Yu=0$ in the sense of distributions, for all such vector fields $Y$.
This implies that $u$ is constant and, therefore, that $u$ vanishes, by square-integrability.
\end{proof}

\begin{Theorem}
Suppose that $K_M<0$ and that $M$ is asymptotically harmonic about a point $\xi\in M_\infty$.
Let $D\subseteq M$ be the shadow of a smooth domain $C$ in a horosphere $H$ in $M$ with center~$\xi$.
Then the Dirichlet and Neumann spectra of $\Delta$ on $D$ are absolutely continuous.
\end{Theorem}

\begin{proof}
The proof is similar to the above one,
changing the proof of Theorem~\ref{nopointd8}, instead of Theorem~\ref{nopointd}, analogously.
\end{proof}

\begin{proof}[Proof of first part of Theorem~\ref{onlyacm}]
Let $b$ be a Busemann function associated to $\xi$ and ${X=\nabla b}$.
Then $|X|=1$ and $\diver X=h>0$.
Furthermore, $\nabla X$ is positive definite on the orthogonal complement of $X$,
that is, in the direction to the horospheres in $M$ with center $\xi$.

Let $Y$ be a smooth vector field on $M$ with compact support such that $Y\perp X$.
Then there are constants $C_1,C_2>0$ such that $\nabla X\ge C_1$ on the support of $Y$ and perpendicularly to $X$
and such that $|Y|^2\le C_2$.

Now $C^\infty_{\rm c}(M)$ is a core for $\Delta$.
Since $\diver X=h>0$ is constant, \eqref{x1} and \eqref{x2a} yield, for $u\in C^\infty_{\rm c}(M)$,
\begin{align*}
 \int_M (\Delta u-\lambda u) P_Xu
 &= \int_M \langle\nabla_{\nabla u}X,\nabla u\rangle \ge C_1 \int_M|\nabla^\bot u|^2 \\
 &\ge \frac{C_1}{C_2}\int_M|Yu|^2 = \frac{C_1}{C_2}\|T_Yu\|_2^2,
\end{align*}
where $\nabla^\bot u$ denotes the component of $\nabla u$ perpendicular to $X$.
Now \eqref{x3} applies and shows that $T_Y$ is $S$-smooth on any given bounded Borel subset $A\subseteq\R$.
Hence \eqref{Kato3} implies that $\ran E_AT_Y^*\subseteq H_{\ac}(\Delta)$,
for any bounded Borel subset $A\subseteq\R$.
Therefore, $\ran T_Y^*\subseteq H_{\ac}(\Delta)$.

Now suppose that $u\in L^2(D )$ is perpendicular to $H_{\ac}(\Delta)$.
Then $u$ is perpendicular to $\ran T_Y^*$, for all vector fields $Y$ as above.
But then $T_Yu=0$ in the sense of distributions, for all such vector fields $Y$.
This implies that $u$ is constant along horospheres, that is, levels of $u$ are unions of horospheres.
Now the preimage $b^{-1}(B)$ of any Borel subset $B\subseteq\R$ of positive measure has infinite measure.
Therefore, $u$ vanishes, by square-integrability.
\end{proof}

\begin{proof}[Proof of part of Theorem~\ref{onlyacn}]
By the definition of elementary groups of isometries of $M$, the gradient field $X$ of Busemann functions,
that is, the field of velocity vectors of unit speed geodesics emanating from $\xi$,
is invariant under $\Gamma$.
Notice that $X$ spans the normal space to the fibers of the Riemannian submersion $\pi$.
Now the arguments of the previous proof apply and show that $u\in L^2(N)$ is perpendicular to $H_{\ac}(N)$
if and only if $u$ is constant on the fibers of $\pi$.
There are now two cases:
The fibers of $\pi$ have infinite volume.
This is always the case when $\Gamma$ is of type \ref{ax}.
Then, as in the previous proof, the square-integrability of $u$ implies that $u=0$.

In the second case, $\Gamma$ is of type \ref{pa}.
Then the flow $(F_t)_{t\in\R}$ of $X$ induces diffeomorphisms between the fibers of $\pi=b$,
and the volume element of horospheres is multiplied by $\exp(ht)$ under $F_t$.
Hence, since the fibers of $b$ have finite volume, their volumes satisfy
\begin{align*}
	\exp(ht) \big|b^{-1}(s)\big| = \exp(hs) \big|b^{-1}(t)\big|.
\end{align*}
Furthermore, the space $L^2_b(N)$ of functions on $N$, which are constant fiberwise,
and its orthogonal complement $L^2_0(N)$ are invariant under $\Delta$.
That is, the spectrum of $\Delta$ is the combination of the spectra of $\Delta$ on $L^2_b(N)$
and of $\Delta$ on $L^2_0(N)$.
By what we said above, if $u\in L^2_0(N)$ is perpendicular to $H_{\ac}(N)$, then $u=0$.
Thus $L^2_0(N)\subseteq H_{\ac}(N)$, and we are left with discussing $\Delta$ on $L^2_b(N)$.

Note that $u\in L^2(N)$ is in $L^2_b(N)$ if and only if $u$ is a pull-back $b^*v$ for some function $v$ on $\R$.
We endow $\R$ with the measure $\mu=h_0\exp(ht){\rm d}t$,
where $h_0=\big|b^{-1}(0)\big|$ and get that
\begin{align*}
	b^* \colon \ L^2(\R,\mu) \to L^2_b(N)
\end{align*}
is a unitary transformation.
With respect to $b^*$, the Laplacian on $N$ corresponds to a diffusion operator on $\R$,
\begin{align*}
	\Delta b^*v = b^*\bigl(- v'' - hv'\bigr).
\end{align*}
Let $\vf$ be the square root of $h_0\exp(ht)$.
Then
\begin{align*}
	\vf' = \frac{h}2\vf \qquad\text{and}\qquad \vf'' = \frac{h^2}4\vf.
\end{align*}
Therefore, under the unitary transformation
\begin{align*}
	L^2(\R,\mu) \to L^2(\R), \qquad v \mapsto \vf v,
\end{align*}
we obtain
\begin{align*}
	(\vf v)'' = \vf''v +2\vf'v' + \vf v'' = \vf\biggl(\frac{h^2}4 v + hv' + v''\biggr).
\end{align*}
We conclude that the Schr\"odinger operator $Sw=-w''+h^2w/4$ in $L^2(\R)$
corresponds to the above diffusion operator on $L^2(\R,\mu)$.
Since $h^2/4$ is a constant, the spectrum of the latter is absolutely continuous and equal to $\bigl[h^2/4,\infty\bigr)$.
By unitary equivalence, the same holds for $\Delta$ on~$L^2_b(N)$.

By Observation~\ref{chee}, $\sigma(N)\subseteq\bigl[h^2/4,\infty\bigr)$, hence the above shows equality,
$\sigma(N)=\bigl[h^2/4,\infty\bigr)$, in the case where the fibers of $\pi$ are of finite volume.
The case where the fibers are of infinite volume will be discussed in Section~\ref{secsig}.
\end{proof}

\section{Homogeneous Hadamard manifolds}
\label{sechom}

The aim of this section is the proof of Theorem~\ref{onlyach}.
In contrast to Theorem~\ref{onlyacm}, we do not assume that the sectional curvature is negative.
For that reason, we need to investigate the underlying geometry of the manifolds more carefully
and start with the necessary background material.
Our main source is~\cite{Wolter91},
and we refer the reader to it for more details and references.
Another good reference is~\cite{Heber93}.

Let $M$ be a homogeneous Hadamard manifold.
Then there is a solvable Lie group $S$ of isometries of $M$ which acts simply transitively on $M$.
By choosing an origin $x_0$ of $M$,
the orbit map $S\to M$, $g\mapsto gx_0$, is a diffeomorphism,
and the pull-back of the Riemannian metric of~$M$ to~$S$ is a left-invariant metric on $S$.
In this way, we identify $M$ with the simply connected solvable Lie group $S$,
endowed with the above left-invariant Riemannian metric.

Let $\mathfrak s$ be the Lie algebra of $S$,
Then $\mathfrak n=[\mathfrak s, \mathfrak s]$ is nilpotent,
the orthogonal complement $\mathfrak a=\mathfrak n^\perp$ of $\mathfrak n$ in $\mathfrak s$ is Abelian,
and $S=A\ltimes N$,
where $A$ and $N$ denote the connected Lie subgroups of~$S$ tangent to $\mathfrak a$ and $\mathfrak n$,
respectively.
Note that $A$ and $N$ are simply connected and, therefore, non-compact.
By Corollary~\ref{eufac}, we can assume that $M$ respectively $S$ has no Euclidean factor.
This is convenient since it keeps the structure of $S$ more accessible; cf.~\cite[Section~1]{Wolter91}.

For $v\in\mathfrak a$,
let $D_v$ and $S_v$ be the symmetric and skew-symmetric parts of $\ad_v|_{\mathfrak n}$.

\begin{Proposition}
There are an $\ad\mathfrak a$-invariant orthogonal decomposition
\begin{align*}
	\mathfrak n = \mathfrak n_1 \oplus\dots\oplus \mathfrak n_k
\end{align*}
 of $\mathfrak n$ into subalgebras $\mathfrak n_j$, \emph{roots} $\mu_j\in\mathfrak a^*$,
 and positive definite symmetric operators $D_j$ on $\mathfrak n_j$, for $1\le j\le k$, such that
\begin{enumerate}[label={\rm (\arabic*)}]\itemsep=0pt
\item\label{hh1}
for all $v\in\mathfrak a$, the restriction of the symmetric part $D_v$ of $\ad_v$ to $\mathfrak n_j$
is given by $\mu_j(v)D_j$;
\item
the $\mu_j$ are pairwise not collinear and span $\mathfrak a^*$;
\item
$W=\{v\in\mathfrak a \mid \text{$\mu_j(v)>0$ for all $1\le j\le k$} \}\ne\varnothing$.
\end{enumerate}
\end{Proposition}

Let $v\in W$ and $\gamma$ be the geodesic through the neutral element of $S$ with initial velocity $v$.
By~\cite[Proposition~3.3]{Wolter91}, $S$ fixes $\xi=\gamma(\infty)$.
Hence, since all left-translations are isometries of $S$, we conclude that $S$ is asymptotically harmonic about $\xi$.

View elements of $\mathfrak s$ now as left-invariant vector fields on $S$,
where we use $V$ to denote the one corresponding to $v$.
Since $S$ fixes $\xi$, the gradient of Busemann functions on $S$ centered at $\xi$ equals $-V$.
Hence the second derivative of Busemann functions on $S$ centered at $\xi$ equals $-\nabla V$.

By~\ref{hh1} and the choice of $v$, there exists a constant $c>0$ such that
\begin{align*}
	\langle [v,z],z\rangle = \langle\ad_vz,z\rangle = \langle D_vz,z\rangle \ge c|z|^2
\end{align*}
for all $z\in \mathfrak n$.
Hence we obtain
\[
	-\langle\nabla_ZV,Z\rangle = \langle[V,Z],Z\rangle \ge c|Z|^2
\]
for any left-invariant vector field $Z$ coming from $\mathfrak n$.
Since $\mathfrak a$ is Abelian,
\[
	-\langle\nabla_YV,Y\rangle = \langle[V,Y],Y\rangle = 0
\]
for any left-invariant vector field $Y$ coming from $\mathfrak a$.
Finally, for $Z$ coming from $\mathfrak n$ and $Y$ from $\mathfrak a$, we compute
\[
	\langle \nabla_ZV ,Y \rangle + \langle \nabla_YV ,Z \rangle
= \langle \nabla_VZ ,Y \rangle + \langle \nabla_YV ,Z \rangle = \langle Z , [Y,V] \rangle = 0,
\]
where we use that $\mathfrak n$ is an ideal and that $\mathfrak a$ is Abelian.

\begin{proof}[Proof of first part of Theorem~\ref{onlyach}]
Now we adapt the arguments from the proof of Theorem~\ref{onlyacm} to the present situation and set $X=-V$.
Recall that $V$ is a gradient field so that~$\nabla X$ is symmetric.
On the left-invariant distribution $S\mathfrak a$ coming from $\mathfrak a$, we have $\nabla X=0$.
With~$C_1=c$, where $c$ is as above, we have $\nabla X\ge C_1$
on the left-invariant distribution $S\mathfrak n$ coming from $\mathfrak n$.
Furthermore, $\langle\nabla_{S\mathfrak a}X,S\mathfrak n\rangle=0$.

Let $Y$ be a smooth vector field on $S$ with compact support which is perpendicular to $S\mathfrak a$.
Choose a constant $C_2>0$ such that $|Y|^2\le C_2$.

Now $C^\infty_{\rm c}(M)$ is a core for $\Delta$.
Since $\diver X=h>0$ is constant, \eqref{x1} and \eqref{x2a} yields, for $u\in C^\infty_{\rm c}(M)$,
\begin{align*}
 \int_M (\Delta u-\lambda u) P_Xu
 &= \int_M \langle\nabla_{\nabla u}X,\nabla u\rangle \ge C_1 \int_M|\nabla^\bot u|^2 \\
 &\ge \frac{C_1}{C_2}\int_M|Yu|^2 = \frac{C_1}{C_2}\|T_Yu\|_2^2,
\end{align*}
where $\nabla^\bot u$ denotes the component of $\nabla u$ perpendicular to $S\mathfrak a$.
Now \eqref{x3} applies and shows that $T_Y$ is $S$-smooth on any given bounded Borel subset $B\subseteq\R$.
Hence \eqref{Kato3} implies that $\ran E_BT_Y^*\subseteq H_{\ac}(M)$,
for any bounded Borel subset $B\subseteq\R$.
Therefore, $\ran T_Y^*\subseteq H_{\ac}(M)$.

Now suppose that $u\in L^2(M)$ is perpendicular to $H_{\ac}(M)$.
Then $u$ is perpendicular to $\ran T_Y^*$, for all vector fields $Y$ as above.
But then $T_Yu=0$ in the sense of distributions, for all such vector fields $Y$.
This implies that $u$ is constant along the orbits of $N$,
since the orbits of~$N$ meet the distribution $S\mathfrak a$ orthogonally.
Since $N$ is non-compact, the $N$-orbits have infinite ($\dim N$-dimensional) volume.
Hence the union of orbits meeting a Borel subset of an $A$-orbit of positive ($\dim A$-dimensional) volume
has infinite volume in $S$.
Since $u$ is square-integrable, we conclude that $u$ is identically zero.
This concludes the proof that the spectrum of $M$ is absolutely continuous.
\end{proof}

The second part of Theorem~\ref{onlyach} holds under much more general assumptions.
Therefore, we state it as an extra result here.

\begin{Theorem}\label{specsn}
Let $G$ be a connected Lie group, endowed with a left-invariant Riemannian metric.
Assume that $G$ is solvable or is a compact extension of a solvable Lie group,
Let $N$ be the derived subgroup of $G$, and assume that $Q=N\backslash G$ is non-compact.
Then $\sigma(G)=\bigl[h^2/4,\infty\bigr)$, where $h$ is the mean curvature of $N$ as a submanifold of $G$,
given by
\begin{align*}
	h^2 = \sum_j (\tr\ad_{X_j}|_{\mathfrak n})^2
	= \sum_j (\tr\ad_{X_j})^2.
\end{align*}
Here $(X_j)$ is an orthonormal basis of the orthogonal complement of $\mathfrak n$ in $\mathfrak s$.
\end{Theorem}

\begin{proof}
We invoke Theorem~\ref{p13}.
The Riemannian manifold $M$ there corresponds to the Lie group $G$ here,
the Lie group $G$ acting on $M$ there corresponds to $N$ here, acting freely on $G$ by left-translations.
Since $N$ is a closed subgroup of $G$ and the Riemannian metric on $G$ is left-invariant,
the action is proper and isometric.

Since $N$ is the derived subgroup of $G$, it is a closed and normal subgroup of $G$
such that the quotient $Q=N\backslash G$ is an Abelian Lie group.
As the derived subgroup of a connected Lie group, $N$ is connected.

Since $N$ is normal in $G$, the fibration $\pi\colon G\to Q$ is $G$-invariant.
Therefore, the left-invariant Riemannian metric on $G$ induces a left-invariant Riemannian metric on $Q$,
such that $\pi$ is a~Riemannian submersion.
Since $Q$ is connected and Abelian, the left-invariant metric on $M$ is flat
and $Q$ is a Euclidean space $E$ times a torus $T$.
Since $Q$ is non-compact, $\dim E>0$.
From Corollary~\ref{sumdel} and Example~\ref{rmac},
we conclude that the Laplace spectrum $\sigma(Q)=[0,\infty)$.

Since $\pi$ is invariant under left-translation by $G$,
the field $H$ of mean curvature of the fibers is left-invariant and is $\pi$-related
to a left-invariant vector field on $Q$, the push-forward~$\pi_*H$ of~$H$.
Since $G$ is amenable, being a compact extension of a solvable Lie group,
we derive that so is the closed subgroup $N$.
Furthermore, as the derived subgroup of a Lie group, $N$ is unimodular.
Thus we may apply Theorem~\ref{p13} to obtain that $\sigma(M)=[\lambda_0,\infty)$,
where $\lambda_0$ is equal to the bottom of the spectrum of the Schr\"odinger operator $S$ on $Q$ defined in \eqref{pdef}.

Since $Q$ is Abelian and $\pi_*H$ is left-invariant,
$\pi_*H$ is parallel, hence of constant norm with vanishing divergence.
Therefore, the potential of the Schr\"odinger operator $S$ in \eqref{pdef} is constant
and equal to $h^2/4=|\pi_*H|^2/4=|H|^2/4$.
Since the bottom of the spectrum of the Laplacian on~$Q$ is equal to~$0$,
we get $\lambda_0=h^2/4$ as claimed.

Let $(Y_i)$ be an orthonormal basis of $\mathfrak n$.
Then the mean curvature field $H$ of the fibers of~$\pi$ is the left-invariant vector field
which, at the neutral element, is equal to the component of $\sum_i\nabla_{Y_i}Y_i$
perpendicular to $\mathfrak n$.
If $(X_j)$ is an orthonormal basis of the orthogonal complement of~$\mathfrak n$ in~$\mathfrak s$,
then we get that
\begin{align*}
	H = \sum_{i,j}\langle\nabla_{Y_i}Y_i,X_j\rangle X_j
	= \sum_{i,j}\langle[X_j,Y_i],Y_i\rangle X_j
	= \sum_j(\tr\ad_{X_j}|_{\mathfrak n})X_j,
\end{align*}
which implies the first asserted equality.
The second follows since the orthogonal complement of $\mathfrak n$ in $\mathfrak s$
does not contribute to $\tr\ad_{X_j}$.
\end{proof}

\section{Determination of the spectrum}\label{secsig}

In this section we aim at the determination of the spectrum $\sigma(M)$.
We start with the case where $M$ is a Hadamard manifold with sectional curvature $K_M<0$
which is asymptotically harmonic about a point $\xi\in M_\infty$
with mean curvature of horospheres equal to $h>0$.
Recall from the first part of the proof of Theorem~\ref{onlyacn}
that the case left is $N=\Gamma\backslash M$,
where the fibers of the associated Riemannian submersion $\pi=b\colon N\to\R$ are of infinite volume,
where $b$ is the push-down to $N$ of some Busemann function associated to $\xi$.
We subsume the case $N=M$ here by including $\Gamma=\{\id\}$.
From Observation~\ref{chee}, we know that $\sigma(N)\subseteq\bigl[h^2/4,\infty\bigr)$.

For any smooth function $v$ on $\R$, we have
\begin{align*}
	\Delta b^*v = b^*\bigl(-v'' - hv'\bigr).
\end{align*}
Now $\vf=\vf(t)=\exp(ht/2)$ satisfies $\vf'=h\vf/2$ and $\vf''=h^2\vf/4$.
With $\mu=\vf^2{\rm d}t$, the unitary transformation
\begin{align*}
	\Phi \colon \ L^2(\R,\mu) \to L^2(\R), \qquad \Phi(v) = \vf v,
\end{align*}
yields
\begin{align*}
	(\Phi(v))'' = (\vf v)'' = \vf\biggl(\frac{h^2}4v + hv' + v''\biggr),
\end{align*}
and hence $\Phi$ intertwines the diffusion operator $Lv=-(v''+hv')$ in $L^2(\R,\mu)$
with the Schr\"odinger operator $Sw=-w''+h^2w/4$ in $L^2(\R)$.
Therefore, the spectra of $L$ and $S$ coincide,
where we recall that $H_{\ac}(S)=L^2(\R)$ with $\sigma(S)=\bigl[h^2/4,\infty\bigr)$.

We want to show that $\sigma(S)\subseteq\sigma(N)$.
In contrast to the case where the volume of the fibers of $b$ are of finite volume,
we cannot simply use pull backs $b^*v$ since we would lose square-integrability that way.
What we will do is to carefully cut off pull-backs.

\begin{Lemma}
Suppose that $-1\le K_M\le-a^2<0$ and that $\|\nabla R_M\|_\infty<\infty$.
Let $\chi_0\colon b^{-1}(0)\to\R$ be a non-zero $C^2$ function with compact support.
Extend $\chi_0$ to $N$ by letting $\chi=\chi_t$ be equal to~$F_{-t}^*\chi_0$ on $b^{-1}(t)$,
where $(F_t)_{t\in\R}$ denotes the flow of $X=\nabla b$.
Then $\chi$ is $C^2$ and $\nabla\chi$ and $\Delta\chi$ tend to zero uniformly as $t\to\infty$.
\end{Lemma}

\begin{proof}
Since $\chi$ is constant along the flow lines of $X$,
the gradient of $\chi$ is tangential to the fibers of $b^{-1}(t)$ of $b$.
By the upper curvature bound $-a^2$, unstable Jacobi field $J$ grow at least exponentially,
\begin{align*}
	|J(t)| \ge {\rm e}^{at} |J(0)| \qquad\text{for all} \quad t>0.
\end{align*}
Therefore,
\begin{align*}
	|\nabla\chi_t| \le {\rm e}^{-at} |\nabla\chi_0| \qquad\text{for all}\quad t>0,
\end{align*}
which implies the first assertion.

For the estimate of the Laplacian, we refer to~\cite[Section 6]{BP21b}.
The situation there is that of a Hadamard manifold $X$ with pinched negative sectional curvature
and uniformly bounded covariant derivative of its curvature tensor,
a convex domain $C$ in $X$,
and a function obtained by extending a given function on $\partial C$
to $X\setminus C$ constantly along minimizing geodesics to $C$.
In our situation, $X$ corresponds to $M$, $C$ to the horoball $b^{-1}((-\infty,0])$,
and the given function to $\chi_0$.
A short look at the arguments in~\cite[Section 6]{BP21b} shows
that they also apply in the corresponding situation in $N$.
The estimate obtained in~\cite{BP21b} is that
\begin{align*}
	|\Delta\chi| \le c_0{\rm e}^{-at} \qquad\text{over} \quad b^{-1}(t)\quad \text{for all} \quad t>0,
\end{align*}
where $c_0$ is a constant, which depends on bounds for $K_M$, $\nabla R_M$, $\nabla\chi_0$, and $\nabla^2\chi_0$;
compare with~\cite[formula~(6.4) and end of Section~6]{BP21b}.
\end{proof}

\begin{proof}[End of proof of Theorem~\ref{onlyacn}]
Let $\lambda\in\sigma(S)$ and $\ve>0$.
Choose a non-vanishing $w\in C^\infty_{\rm c}(\R)$ such that $\|(S-\lambda)w\|_2\le\ve\|w\|_2$.
By shifting $w$ along $\R$ if necessary, we can assume that $\supp w\subseteq[t,\infty)$,
where $t>0$ is such that $|\Delta\chi| < \ve$ on $b^{-1}([t,+\infty))$.
Note that shifting $w$ is legitimate since $Sw$ is only shifted accordingly.
Then $v=w/\vf$ also has support in $[t,\infty)$ and satisfies $\|(L-\lambda)v\|_2\le\ve\|v\|_2$,
where now the $L^2$-norm is taken with respect to $\mu=\vf^2{\rm d}t$.
Then
\begin{align*}
	\Delta(\chi b^*v)
	= (\Delta\chi)b^*v - 2\langle\nabla\chi,\nabla b^*v\rangle + \chi\Delta b^*v
	= (\Delta\chi)b^*v + \chi b^*Lv
\end{align*}
since $\nabla\chi$ is perpendicular to $X$ and $\nabla b^*v$ is collinear with $X$
and since $b^*$ intertwines $\Delta$ with $L$.

Since $b$ is a Riemannian submersion and the flow maps $F_s$ of $X$ induces
a diffeomorphism of~$b^{-1}(0)$ with $b^{-1}(s)$,
for all $s\in\R$ and with respective Jacobian $\vf^2(s)={\rm e}^{hs}$.
Thus we identify~$x\in N$ with $(y,s)\in b^{-1}(0)\times\R$, where $b(x)=s$ and $F_s(y)=x$,
and get
\begin{align*}
	\|\chi b^*v\|_2^2
	&= \int \chi^2(x)(b^*v(x))^2{\rm d}x \\
	&= \iint_{b^{-1}(0)\times\R}\chi_0^2(y)v^2(s)\vf^2(s){\rm d}y{\rm d}s \\
	&= \|\chi_0\|_2^2 \int_\R v^2(s)\vf^2(s){\rm d}s
	= \|\chi_0\|_2^2 \|v\|_2^2,
\end{align*}
where $\|\chi_0\|_2$ denotes the $L^2$-norm of $\chi_0$ in $L^2\bigl(b^{-1}(0)\bigr)$.
Likewise, by the choice of $t$,
\begin{align*}
	\|(\Delta\chi)b^*v\|_2^2
	\le \ve^2|\supp\chi_0| \|v\|_2^2.
\end{align*}
and hence
\begin{align*}
	\|(\Delta\chi)b^*v\|_2^2
	\le \ve^2\frac{|\supp\chi_0|}{ \|\chi_0\|_2^2}\|\chi b^*v\|_2^2
	= \ve^2c_1^2\|\chi b^*v\|_2^2,
\end{align*}
where $c_1$ depends on the choice of $\chi_0$.
Hence we obtain
\begin{align*}
	\|(\Delta-\lambda)(\chi b^*v)\|_2^2
	&\le 2\|\chi b^*(L-\lambda)v\|_2^2 + 2\ve^2c_1^2\|\chi b^*v\|_2^2 \\
	&= 2\|\chi_0\|_2^2 \|(L-\lambda)v\|_2^2 + 2\ve^2c_1^2\|\chi b^*v\|_2^2 \\
	&= \ve^2c_2^2\|\chi b^*v\|_2^2,
\end{align*}
where $c_2$ depends on the choice of $\chi_0$.
This implies that $\lambda\in\sigma(N)$.
\end{proof}

\appendix

\section{Some spectral theory}\label{secspec}

We include a short discussion of spectra of self-adjoint operators and of Riemannian manifolds.
The main aim is the description of the spectra of Riemannian products.
The latter is well-known and clear in the compact case,
but we could not single out a reference for the general case as presented in Appendix~\ref{subrie}.
The product case is important in our discussion of homogeneous Hadamard manifolds,
where it is used for the reduction to the case of vanishing Euclidean factor,
which keeps the technicalities of the discussion at a moderate level; cf.\ Section~\ref{sechom}.

\subsection{Convolution of measures}
Recall that a finite Borel measure $\nu$ on $\R$ is uniquely a sum of two finite Borel measures $\nu_{\ac}$ and $\nu_{\s}$ on $\R$,
where $\nu_{\ac}$ is absolutely continuous and $\nu_{s}$ is singular with respect to Lebesgue measure $\lambda$.
Then $\R$ is the disjoint, but not unique, union of a Lebesgue measurable set $Y_{\ac}$ and a Lebesgue nullset $Y_{\s}$ such that $\nu_{\s}$ is concentrated on $Y_{\s}$,
that is, $\nu_{\s}(\R\setminus Y_{\s})=0$.
The countable set $Y_{\p}=\{y\in\R\mid\nu(y)>0\}$ is a subset of $Y_{\s}$
and $\nu_{\s}=\nu_{\p}+\nu_{\sc}$,
where $\nu_{\p}$ is concentrated on~$Y_{\p}$ and $\nu_{\sc}$ on~$Y_{\sc}=Y_{\s}\setminus Y_{\p}$.

The \emph{convolution} $\nu_1*\nu_2$ of finite Borel measures $\nu_1$ and $\nu_2$ on $\R$ is defined to be the push-forward of the finite Borel measure $\nu_1\otimes\nu_2$ on $\R^2$ under the map
\begin{align*}
	\R^2 \to \R, \qquad (x_1,x_2)\mapsto x_1+x_2.
\end{align*}
In other words, for any Borel set $B\subseteq\R$,
\begin{align*}
	(\nu_1*\nu_2)(B) &= (\nu_1\otimes\nu_2)(\{(x_1,x_2)\mid x_1+x_2\in B\}) \\
	&= \iint\chi_B(x_1+x_2)\, {\rm d}\nu_1(x_1){\rm d}\nu_2(x_2).
\end{align*}
Clearly $\nu_1*\nu_2=\nu_2*\nu_1$.
Furthermore, if $\nu_1=g_1\lambda$ with a $\lambda$-integrable $g_1$ on $\R$, then
\begin{align}
	\nu_1*\nu_2 &= g\lambda \qquad\text{with $\lambda$-integrable function}\quad g(y)=\int g_1(y-z)\,{\rm d}\nu_2(z).\label{con1}
\intertext{In particular, if in addition $\nu_2=g_2\lambda$ with a $\lambda$-integrable $g_2$ on $\R$, then}
	\nu_1*\nu_2 &= g\lambda \qquad\text{with $\lambda$-integrable function}\quad g(y)=\int g_1(y-z)g_2(z)\,{\rm d}\lambda(z).\label{con2}
\end{align}
Notice that the latter $g=g_1*g_2$, the usual convolution of $g_1$ and $g_2$.

The support of a measure $\nu$ on $\R$ is the closed set of points $y\in\R$ such that $\nu(U)>0$
for any neighborhood $U$ of $y$.

\begin{Proposition}\label{consupp}
The convolution $\nu=\nu_1*\nu_2$ has support
\begin{align*}
	\supp\nu = \overline{\supp\nu_1+\supp\nu_2}.
\end{align*}
If $\supp\nu_1$ and $\supp\nu_2$ are bounded from below,
then $\supp\nu = \supp\nu_1+\supp\nu_2$.
\end{Proposition}

For Borel measures $\mu,\nu$ on $\R$,
we write $\mu\prec\nu$ to indicate that $\mu$ is absolutely continuous with respect to $\nu$.

\begin{Proposition}\label{consum}
The components of the convolution $\nu=\nu_1*\nu_2$ satisfy
\begin{gather*}
	\nu_{\p} = \nu_{1\p}*\nu_{2\p} \qquad\text{with}\quad Y_{\p}=Y_{1\p}+Y_{2\p} ; \\
	\nu_{\ac} \succ \nu_{1\ac}*\nu_{2\ac} + \nu_{1\ac}*\nu_{2\s} + \nu_{1\s}*\nu_{2\ac}.
\end{gather*}
In particular, if $\nu_1$ or $\nu_2$ is absolutely continuous with respect to $\lambda$, then also $\nu$.
\end{Proposition}

In general,
we cannot exclude that the absolutely continuous component of the convolution of $\lambda$-singular measures vanishes.
Thus $\nu_{1\s}*\nu_{2\s}$ may contribute to $\nu_{\ac}$.

\begin{proof}[Proof of Proposition~\ref{consum}]
The assertion about $\nu_{\p}$ is clear from
\begin{align*}
	(\nu_1*\nu_2)(y) = \sum_{y_1+y_2=y}\nu_1(y_1)\nu_2(y_2).
\end{align*}
The assertion about $\nu_{\ac}$ follows immediately from \eqref{con1} and \eqref{con2}.
\end{proof}

\subsection{Decomposition of spectra}
Let $A$ be a self-adjoint operator on a (real or complex, separable) Hilbert space $H$.
A second such operator $A'$ on a Hilbert space $H'$ is viewed as equivalent to $A$
if there is an orthogonal respectively unitary transformation $T\colon H\to H'$
such that $T\dom A=\dom A'$ and $TA=A'T$ on $\dom A$.
We say that such an operator $T$ \emph{intertwines} $A$ and $A'$.
The spectral theorem in its multiplicative version says that there is a measure space $X$
with a finite measure $\mu$ and a~measurable real function $f$ on $X$
such that $A$ is equivalent to the multiplication operator~$A_f$ on $L^2(X,\mu)$, $A_f\vf=f\vf$,
with domain $\dom A_f=\{\vf\in L^2(X)\mid f\vf\in L^2(X)\}$.
Note that neither~$X$ nor $\mu$, given $X$, are unique.

Suppose that $T$ intertwines $A$ with such a multiplication operator $A_f$.
Consider the push-forward $\nu=f_*\mu$, a Borel measure on $\mathbb{R}$.
By the definition of $f_*\mu$,
\begin{align*}
	\supp\nu = \essran f,
\end{align*}
the \emph{essential range of $f$}.
Whereas $\nu$ is not unique, not even the total mass of $\nu$,
the spectrum~$\sigma (A)$, the spectral projections $\pi_B$, and the spectral measures $\nu_u$, $u\in H$,
associated to $A$ are invariant respectively equivariant under $T$.
For that reason, it will be instructive to identify them in the case of $A_f$.

Recall that the \emph{resolvent set} $\rho(A)$ of $A$ consists of all $y\in\R$ or $y \in \C$, respectively, such that
\begin{align*}
	A-y \colon \ \dom A \to H
\end{align*}
is bijective.
By definition, the \emph{spectrum} $\sigma(A)$ is the complement of $\rho(A)$ in $\R$ or $\C$, respectively,
where we actually have $\sigma(A)\subseteq\R$, by the self-adjointness of $A$.
The \emph{point spectrum} ${\sigma_{\p}(A)\subseteq\sigma(A)}$ is the set of eigenvalues of $A$.

\begin{Proposition}\label{aaf}
In the notation further up,
\begin{align*}
	\sigma(A) = \sigma(A_f) = \essran f
	\qquad\text{and}\qquad
	\sigma_{\p}(A)=\sigma_{\p}(A_f)=\{y\in\R\mid\nu(y)>0\}.
\end{align*}
\end{Proposition}

\begin{Proposition}
The spectral projection of $A_f$ associated to a Borel set $B\subseteq\R$
is given by multiplication in $L^2(X)$ with $f^*\chi_B=\chi_B\circ f$,
where $\chi_B$ denotes the characteristic function of $B$.
\end{Proposition}

For any $u\in H$,
the \emph{spectral measure} $\nu_u$ on $\R$ associated to $u$ (and $A$) is given by
\begin{align*}
	\nu_u(B) = \langle u,\pi_Bu\rangle_2
\end{align*}
for any Borel set $B\subseteq\R$,
where $\pi_B\colon H\to H$ denotes the spectral projection associated to $A$.
Since $T$ intertwines the spectral measures associated to $A$ and $A_f$,
\begin{align*}
	\nu_u(B) = \nu_\vf(B) = \langle\vf,(f^*\chi_B)\vf\rangle_2 = \int_X (f^*\chi_B)\bar\vf\vf\, {\rm d}\mu
\end{align*}
for $\vf=Tu\in L^2(X)$ and any Borel set $B\subseteq\R$.

Let now again $\nu=f_*\mu$, decompose $\nu=\nu_{\ac}+\nu_{\p}+\nu_{\sc}$,
and write $\R=Y_{\ac}\cup Y_{\s}$ and $Y_{\s}=Y_{\p}\cup Y_{\sc}$ as disjoint unions as further up.
Then we have, respectively set,
\begin{gather*}
	\sigma(A) = \sigma(A_f) = \essran f,\\
	\sigma_{\ac}(A) =\sigma_{\ac}(A_f) = Y_{\ac}\cap\essran f,\\
	\sigma_{\s}(A) = \sigma_{\s}(A_f) = Y_{\s}\cap\essran f, \\
	\sigma_{\p}(A) = \sigma_{\p}(A_f) = Y_{\p}\cap\essran f = Y_{\p}, \\
	\sigma_{\sc}(A) = \sigma_{\sc}(A_f) = Y_{\sc}\cap\essran f.
\end{gather*}
We call the latter four the \emph{absolutely continuous, singular, point} and \emph{singular continuous}
parts of the spectrum of $A$ or $A_f$.

\begin{Proposition}
Let $L^2(X)_{\ac}$, $L^2(X)_{\s}$, $L^2(X)_{\p}$, $L^2(X)_{\sc}$ be the subspaces of $\vf\in L^2(X)$
vanishing outside $f^{-1}(Y_{\ac})$, $f^{-1}(Y_{\s})$, $f^{-1}(Y_{\p})$, $f^{-1}(Y_{\sc})$, respectively.
Then we have orthogonal decompositions
\begin{align*}
	L^2(X) = L^2(X)_{\ac}\oplus L^2(X)_{\s}
	\qquad\text{and}\qquad
	L^2(X)_{\s} = L^2(X)_{\p}\oplus L^2(X)_{\sc}.
\end{align*}
Moreover, besides zero, these subspaces consist precisely of those $\vf\in L^2(X)$
such that the associated spectral measures $\nu_\vf$ are absolutely continuous and singular with respect to the Lebesgue measure
respectively are concentrated on $\sigma_{\p}(A_f)$ and $\sigma_{\sc}(A_f)$.
\end{Proposition}

\begin{proof}
The only non-trivial part is the absolute continuity of $\nu_\vf$ with respect to the Lebesgue measure $\lambda$
for $\vf\in L^2(X)$ vanishing outside $f^{-1}(Y_{\ac})$.
To that end, let $B \subset \R$ be a Lebesgue nullset.
Since $\vf$ vanishes outside $f^{-1}(Y_{\ac})$, we readily see that
\[
\nu_{\vf}(B) = \int_{f^{-1}(B)} \bar{\vf} \vf\, {\rm d} \mu = \int_{f^{-1}(B \cap Y_{\ac})} \bar{\vf} \vf \,{\rm d} \mu.
\]
From the fact that $\mu\bigl(f^{-1}(B \cap Y_{\ac})\bigr) = \nu(B \cap Y_{\ac}) = \nu_{\ac}(B \cap Y_{\ac}) = 0$
and the square-integrability of $\vf$, we conclude that $\nu_{\vf}(B) = 0$,
which shows that $\nu_{\vf}$ is absolutely continuous with respect to~$\lambda$.
\end{proof}

The characterization of the various subspaces in terms of the corresponding spectral measures
leads under intertwining to a corresponding decomposition of $H$.

\begin{Proposition}
With respect to $A$, we have orthogonal decompositions
\begin{align*}
	H = H_{\ac} \oplus H_{\s} \qquad\text{and}\qquad H_{\s} = H_{\p} \oplus H_{\sc},
\end{align*}
where the associated spectral measures $\nu_\vf$ are absolutely continuous and singular with respect to the Lebesgue measure
respectively are concentrated on $\sigma_{\p}(A)$ and $\sigma_{\sc}(A)$.
\end{Proposition}

For $k=1,2$, let $X_k$ be measure spaces with finite measures $\mu_k$,
Let $f_k$ be measurable functions on $X_k$ and $A_k$ be the self-adjoint operator in $L^2(X_k)$
given by multiplication with $f_k$.
Denote by $\nu_k=f_{k*}\mu_k$ the associated Borel measures on $\R$.
Consider the product $X=X_1\times X_2$ with the product measure $\mu=\mu_1\otimes\mu_2$
and the measurable function $f$ on $X$ defined by
\begin{align*}
 f(x_1,x_2) = f_1(x_1) + f_2(x_2).
\end{align*}
Let $A$ be the self-adjoint operator on $L^2(X)$, given by multiplication with $f$,
and denote by $\nu=f_*\mu$ the associated Borel measure on $\R$.
By the definition of $f$ and of the convolution of measures,
\begin{align}\label{nupro}
	\nu = \nu_1*\nu_2
\end{align}
since $\nu(B)=\mu(\{(x_1,x_2)\mid f_1(x_1)+f_2(x_2)\in B\})$, for any Borel set $B\subseteq\R$.

Recall that $L^2(X)$ is equal to the Hilbert tensor product $L^2(X_1)\hat\otimes L^2(X_2)$.
Accordingly, for $k=1,2$, let $A_k$ be a self-adjoint operator on a Hilbert space $H_k$
and define a self-adjoint operator $A$ on the Hilbert tensor product $H=H_1\hat\otimes H_2$ via
\begin{align*}
	A(u_1\otimes u_2) = (A_1u_1)\otimes u_2 + u_1\otimes(A_2u_2).
\end{align*}
Then Propositions \ref{consupp}--\ref{aaf} and equation \eqref{nupro}
yield the following results.

\begin{Proposition}\label{suma}
In $H=H_1\hat\otimes H_2$, we have
\begin{align*}
	\sigma(A) = \overline{\sigma(A_1)+\sigma(A_2)}.
\end{align*}
If $\sigma(A_1)$ and $\sigma(A_2)$ are bounded from below,
then $\sigma(A)=\sigma(A_1)+\sigma(A_2)$.
\end{Proposition}

\begin{Proposition}
In $H=H_1\hat\otimes H_2$, we have
\begin{gather*}
	H_{\p}(A) = H_{\p}(A_1) \hat\otimes H_{\p}(A_2) \qquad\text{with}\quad
	\sigma_{\p}(A) = \sigma_{\p}(A_1) + \sigma_{\p}(A_2); \\
	H_{\ac}(A) \supseteq
	H_{\ac}(A_1) \hat\otimes H_{\ac}(A_2) \oplus H_{\ac}(A_1) \hat\otimes H_{\s}(A_2)
	\oplus H_{\s}(A_1) \hat\otimes H_{\ac}(A_2).
\end{gather*}
\end{Proposition}

\begin{Corollary}\label{sumc}\quad
\begin{enumerate}[label={\rm (\arabic*)}]\itemsep=0pt
\item If $H_{\p}$ of $A_1$ \emph{or} $A_2$ vanishes, then $H_{\p}(A)=\{0\}$.
\item If $H_{\s}$ of $A_1$ \emph{or} $A_2$ vanishes, then $H_{\s}(A)=\{0\}$.\label{sus}
\item If $H_{\sc}$ of $A_1$ \emph{and} $A_2$ vanishes, then $H_{\sc}(A)=\{0\}$.
\end{enumerate}
\end{Corollary}

Proposition~\ref{suma} is the version of the Ichinose lemma for self-adjoint operators, not necessarily semi-bounded, on Hilbert spaces;
compare with~\cite[Theorem 4.1]{Ich72} and~\cite[Section XIII.9]{ReedSimon78}
and references therein.

\subsection{Spectra of Riemannian manifolds}\label{subrie}
In the present paper, we study Laplacians $\Delta$
of complete and connected Riemannian manifolds $M$ or domains $D$
as self-adjoint operators in their respective spaces of square-integrable functions.

In the case of Riemannian products, $M=M_1\times M_2$,
\begin{align*}
	\Delta(u_1\otimes u_2) = (\Delta u_1)\otimes u_2 + u_1\otimes(\Delta u_2).
\end{align*}
Since the Laplacian is bounded from below, Proposition~\ref{suma} yields

\begin{Corollary}\label{sumdel}
We have $\sigma(M)=\sigma(M_1)+\sigma(M_2)$.
\end{Corollary}

The following example is well-known.

\begin{exa}\label{rmac}
Via the Fourier transform, the Laplacian $\Delta$ of $\R^m$ corresponds to multiplication with $|x|^2$ on $L^2(\R^m)$.
Therefore, the spectrum of $\R^m$ is absolutely continuous with ${\sigma(\R^m)=[0,\infty)}$.
\end{exa}

Together with this example, Corollary~\ref{sumc}\,\ref{sus} yields

\begin{Corollary}\label{eufac}
For any complete and connected Riemannian manifold $M$,
the spectrum of $M\times\R^k$ is absolutely continuous, for any $k\ge1$.
\end{Corollary}

\subsection*{Acknowledgements}

We would like to thank the Max Planck Institute for Mathematics in Bonn for its support and hospitality.
The second author’s research was also partially supported by SEED Grant RD/0519-IRCCSH0-024.
We would like to thank our institutions for providing us with ideal working conditions.
We are also grateful to the referees for their helpful comments.

\pdfbookmark[1]{References}{ref}
\LastPageEnding

\end{document}